\documentclass[a4paper,10pt]{amsart}

\usepackage{a4}
\usepackage{graphics,graphicx}
\usepackage{amssymb,amsmath}
\usepackage[all]{xy}
\usepackage{longtable}
\usepackage{multirow}

\CompileMatrices
\usepackage{hyperref}

\usepackage{tikz}
\usetikzlibrary{calc}
\usetikzlibrary{decorations.markings}
\usetikzlibrary{arrows,shapes,matrix}
\usetikzlibrary{positioning,shapes,shadows,arrows}
\usetikzlibrary{mindmap,trees,calc}
\usetikzlibrary{arrows,shapes,matrix}
\usetikzlibrary{calc}
\usetikzlibrary{arrows,shapes}
\usetikzlibrary{positioning,shapes,shadows,arrows}
\usetikzlibrary{decorations.markings}
\usetikzlibrary{decorations.pathreplacing}
\usepackage{dpfloat}
\usepackage{todonotes}

\newtheorem{theorem}{Theorem}[section]

\newtheorem{proposition}[theorem]{Proposition}

\theoremstyle{definition}

\newtheorem{example}[theorem]{Example}
\newtheorem{remark}[theorem]{Remark}

\newtheorem{construction}[theorem]{Construction}

\newtheorem{algorithm}[theorem]{Algorithm}

\theoremstyle{remark}

\def\tt{\ensuremath{\mathfrak{t}}}

\def\KK{\mathbb{C}}
\def\ZZ{\mathbb{Z}}

\def\QQ{\mathbb{Q}}

\def\OO{\mathcal O}

\def\CC{\mathcal C}

\def\<{\langle}
\def\>{\rangle}

\def\good{/\!\!/}
\def\cox{{\rm Cox}}

\newcommand{\aut}{{\rm Aut}}

\makeatletter
\newcommand{\thickhline}{
    \noalign {\ifnum 0=`}\fi \hrule height 1pt
    \futurelet \reserved@a \@xhline
}
\makeatother

\renewcommand{\phi}{\varphi}
\newcommand{\myhline}{\noalign{\global\arrayrulewidth.95pt}\hline
                      \noalign{\global\arrayrulewidth.2pt}}

\def\b#1{\overline{#1}}

\def\CC{{\mathbb C}}
\def\KK{{\mathbb C}}

\def\ZZ{{\mathbb Z}}

\def\QQ{{\mathbb Q}}

\def\Cox{\cox}

\def\id{{\rm id}}

\def\Aut{\operatorname{Aut}}

\def\Cl{\operatorname{Cl}}

\def\GL{{\rm GL}}

\def\Spec{{\rm Spec}}

\def\Aut{{\rm Aut}}

\def\spec{\Spec}

\def\orig{\Omega}
\def\Orig{\orig}

\def\KT#1{\KK[T_1,\ldots,T_{#1}]}

\def\Stab{{\rm Stab}} 
 
\def\stab{\Stab}

\def\tt#1{\texttt{#1}}
\def\out#1{\begingroup\tiny\begin{gather*} #1 \end{gather*}\endgroup}

\makeatletter
\newcommand{\vast}{\bBigg@{4}}
\newcommand{\Vast}{\bBigg@{5}}
\makeatother

\usepackage{lmodern}
\makeatletter
\ifcase \@ptsize \relax
  \newcommand{\miniscule}{\@setfontsize\miniscule{4}{5}}
\fi
\makeatother

\makeindex

\author[S.~Keicher]{Simon~Keicher}

 \address{Mathematisches Institut, Universit\"at T\"ubingen,
Auf der Morgenstelle 10, 72076 T\"ubingen, Germany}
\email{keicher@mail.mathematik.uni-tuebingen.de}

\title[Compute automorphisms of graded algebras]{A software package to compute automorphisms of graded algebras}

\subjclass[2010]{13P10, 14Q15, 14J50, 13A02, 14L30, 13A50}

\begin{document}

\begin{abstract}
We present a library \texttt{autgradalg.lib} for the free 
computer algebra system \texttt{Singular} 
to compute automorphisms
of integral, finitely generated $\CC$-algebras that are graded pointedly
by a finitely generated abelian group.
It implements the algorithms developed in~\cite{HaKeWo}.
We apply the algorithms to Mori dream spaces and 
investigate the automorphism groups of a series of Fano varieties.
\end{abstract}

\maketitle

\section{Introduction and setting}

Consider an integral, finitely generated $\KK$-algebra
$R$ that is graded by a finitely generated abelian
group $K$, i.e., we have a decomposition
$$
R\ =\ \bigoplus_{w\in K} R_w
\quad
\text{with}
\quad
ff'\in R_{w+w'}
\quad
\text{for all $f\in R_w$, $f'\in R_{w'}$.}
$$
Let the grading to be {\em effective}, i.e.,
the monoid $\vartheta_R \subseteq K$ of all $w\in K$ with $R_w\ne \{0\}$ generates $K$ as a group,
and {\em pointed}: this means that we have $R_0=\KK$ and the polyhedral cone in $K\otimes \QQ$ generated by $\vartheta_R$ is pointed.

We are interested in the {\em automorphism group}
 $\Aut_K(R)$: it consists of all pairs 
 $(\phi,\psi)$ such that $\phi\colon R\to R$ is 
 an automorphism of $\KK$-algebras,
$\psi\colon K\to K$ is an automorphism of groups
and $\phi(R_w)=R_{\psi(w)}$ holds for all $w\in K$.
Note that $\aut_K(R)$ not only is an important invariant of the algebra $R$, 
the methods to compute it can by applied to compute symmetries 
of homogenous ideals~$I$.
Once given explicitely,
the knowledge of the latter largely accelerates further computations involving $I$, 
see~\cite{gfan, gitfansymm, steidel} for examples.

This note presents an implementation \texttt{autgradalg.lib}
of the algorithms from~\cite{HaKeWo}
to compute $\aut_K(R)$.
It is written for the free computer algebra system \texttt{Singular}~\cite{singular}
and is available at~\cite{autmdslib}.
In Section~\ref{sec:autKR}, we describe the algorithm~\cite{HaKeWo} to compute $\aut_K(R)$ and
explain our implementation by a series of examples.
In Section~\ref{sec:MDS}, we apply our implementation to   {\em Mori dream spaces}.
As a result, we determine in Proposition~\ref{prop:fanos} information on the automorphism groups of a class of 
Fano threefolds listed in~\cite{BeHaHuNu}.

\section{Automorphisms of graded algebras}
\label{sec:autKR}

Let us fix the assumptions on the algebra $R$ for our algorithms.
Firstly, we assume the grading group $K$ to be of shape
$\ZZ^k \oplus \ZZ/a_1\ZZ \oplus \ldots \oplus \ZZ/a_l\ZZ$.
In particular, $k$ and the list 
$a_1,\ldots,a_l\in \ZZ_{>1}$ encode~$K$.
The $K$-grading is determined by 
the {\em degree matrix}~$Q=[q_1,\ldots,q_r]$ which has the $q_i := \deg(T_i)$ as its columns. 
Moreover, we expect $R$ to be given explicitly in term
of generators and relations:
$$
R\ =\ S/I,\qquad
S\ :=\ \KT{r}
\qquad
I\ :=\ \<g_1,\ldots,g_s\>\,\subseteq\, S.
$$

As one can remove linear equations, it is no restriction to
assume that $R$ is {\em minimally presented}, i.e., $I\subseteq \<T_1,\ldots,T_r\>^2$ holds
and the generating set $\{g_1,\ldots,g_s\}$ for $I$ is minimal.
From an implementation point of view, it is convenient to 
impose the following slight restrictions: 
\begin{itemize}
 \item 
 the homogenous components $I_{q_1},\ldots,I_{q_r}$ are all trivial,
 \item 
 the set $\{q_1^0,\ldots,q_r^0\}\subseteq \ZZ^k$ of the free parts $q_i^0\in \ZZ^k$ of the $q_i$ contains a lattice basis 
 for~$\ZZ^k$.
\end{itemize}

\begin{example}[\tt{autgradalg.lib} I]
 \label{example:running}
  Consider the following $K:=\ZZ^3\oplus \ZZ/2\ZZ$-graded $\KK$-algebra $R$ from~\cite[Example~2.1]{HaKe, Ke:diss} where
\begin{gather*}
R
\ =\  
S/I,\qquad
S\ :=\ \KT{8},\qquad
I\ :=\ 
\<T_{1}T_{6} + T_{2}T_{5} + T_{3}T_{4} + T_{7}T_{8}\>
,\\
Q \ :=\  
\left[
    \mbox{\tiny $
    \begin{array}{rrrrrrrr}
           1 & 1 & 0 & 0 & -1 & -1 & 2 & -2\\
	   0 & 1 & 1 & -1 & -1 & 0 & 1 & -1\\
	   1 & 1 & 1 & 1 & 1 & 1 & 1 & 1\\
	   \b 1 & \b 0 & \b 1 & \b 0 & \b 1 & \b 0 & \b 1 & \b 0
    \end{array}
    $}
    \right].
\end{gather*}
Then the $K$-grading given by $Q$ is effective
and pointed as hinted in the following picture.
To use \texttt{autgradalg.lib},  
download it from~\cite{autmdslib}
and start \texttt{Singular} in the same directory.
We enter $R$ with the commands\\

\begin{center}
 \begin{minipage}{7cm}
\begingroup
\footnotesize
\begin{enumerate}
\item[\tt >] \tt{LIB "{}gfanlib.so"{}; // for cones}
\item[\tt >] \tt{LIB "{}new\_autgradalg.lib"{};}
\item[\tt >] \tt{intmat Q[4][8] = }
\item[\tt >] \tt{ 1,1,0,0,-1,-1,2,-2,}
\item[\tt >] \tt{ 0,1,1,-1,-1,0,1,-1,}
\item[\tt >] \tt{ 1,1,1,1,1,1,1,1,}
\item[\tt >] \tt{ 1,0,1,0,1,0,1,0;}
\item[\tt >] \tt{list TOR = 2; // torsion part of K}
\item[\tt >] \tt{ring S = 0,T(1..8),dp;}
\item[\tt >] \tt{setBaseMultigrading(Q); // grading}
\end{enumerate}
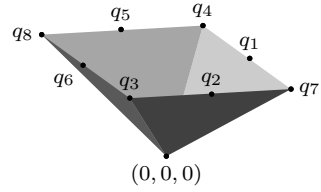
\endgroup
\end{minipage}
 \
 \begin{minipage}{5cm}
   \footnotesize
  \begin{tikzpicture}[scale=2.5]
\coordinate (O) at (0,0,0);
\coordinate (w1) at (0.4,0.48,-.1);
\coordinate (w2) at (0.45,0.54,0.55);
\coordinate (w3) at  (0,0.5,0.5);
\coordinate (w4) at  (0,0.5,-0.5);
\coordinate (w5) at (-0.45,0.46,-0.55);
\coordinate (w6) at (-0.4,0.52,.1);
\coordinate (w7) at (.8,.5,.375);
\coordinate (w8) at (-.8,.5,-.375);
\coordinate (w) at (0,1.1,0);
\coordinate (b1) at (intersection of w8--w2 and w4--w6);
\coordinate (b5) at (intersection of w8--w2 and w6--w7);
\coordinate (b3) at (intersection of w1--w5 and w4--w6);
\coordinate (b4) at (intersection of w4--w2 and w1--w5);
\coordinate (b2) at (intersection of w3--w1 and w2--w6);
\coordinate (b6) at (intersection of w6--w2 and w3--w5);
\coordinate (b7) at (intersection of w3--w5 and w2--w8);
\coordinate (b8) at (intersection of w5--w3 and w6--w4);
\coordinate (a1) at (intersection of w1--w8 and w6--w4);
\coordinate (a2) at (intersection of w1--w8 and w5--w7);
\coordinate (a3) at (intersection of w2--w4 and w5--w7);
\coordinate (a4) at (intersection of w2--w4 and w1--w3);
\coordinate (a5) at (intersection of w1--w3 and w6--w7);
\coordinate (a6) at (intersection of w5--w3 and w6--w7);
\coordinate (a7) at (intersection of w2--w8 and w1--w3);
\coordinate (a8) at (intersection of w2--w4 and w7--w6);
\coordinate (a9) at (intersection of w1--w3 and w5--w7);
\coordinate (d1) at (intersection of w2--w4 and w1--w8);
\coordinate (d2) at (intersection of w5--w7 and w4--w6);
\coordinate (d3) at (intersection of w3--w5 and w1--w8);
\coordinate (d4) at (intersection of w1--w8 and w5--w6);
\coordinate (d5) at (intersection of w2--w8 and w5--w6);
\coordinate (d6) at (intersection of w1--w2 and w6--w7);
\coordinate (d7) at (intersection of w1--w2 and w5--w7);
\fill[color=black!20] (O) -- (w4) -- (w7)--cycle;
\fill[color=black!35] (O) -- (w8) -- (w4)--cycle;
\foreach \w in {1,...,8} {
  \fill (w\w) circle (0.4pt);
}
\fill (0,0,0) circle (0.4pt);
\draw (w1) node[anchor=south]{$q_1$};
\draw (w4) node[anchor=south]{$q_4$};
\draw (w5) node[anchor=south]{$q_5$};
\draw (w6) node[anchor=north east]{$q_6$};
\draw (w7) node[anchor=west]{$q_7$};
\draw (w8) node[anchor=east]{$q_8$};
\draw (O) node[anchor=north]{$(0,0,0)$};
\fill[color=black!75] (O) -- (w3) -- (w7)--cycle;
\fill[color=black!65] (O) -- (w3) -- (w8)--cycle;
\draw (w2) node[anchor=south]{$q_2$};
\draw (w3) node[anchor=south]{$q_3$};
\foreach \w in {1,...,8} {
  \fill (w\w) circle (0.4pt);
}
\fill (0,0,0) circle (0.4pt);
\end{tikzpicture}
 \end{minipage}
\end{center}
\end{example}

Let us recall shortly the steps of the algorithm to compute $\aut_K(R)$; for details, we refer to~\cite{HaKeWo}. 
 The overall idea is to present $\aut_K(R)$ as a stabilizer in the automorphism group $\aut_K(S)$ of the $K$-graded polynomial ring~$S$.
 In a first step, we will compute a presentation  $\aut_K(S)\subseteq \GL(n)$ for some $n\in \ZZ_{\geq 1}$. 
 The set $\Omega_S := \{q_1,\ldots,q_r\}$
 of generator weights will play a major role.
We make use of the following $\GL(n)$-action.
 
 \begin{construction}
\label{con:action}
See~\cite[Construction~3.3]{HaKeWo}.
Write $\Omega_S = \{w_1,\ldots,w_s\}$.
Determine a $\KK$-vector space basis $\mathcal B_i$
for $S_{w_i}$ consisting of monomials.
Then the concatenation $\mathcal B := (\mathcal B_1,\ldots,\mathcal B_s)$
is a basis for $V = \bigoplus_i S_{w_i}$. 
With  $n:=|\mathcal B|$, in terms of $\mathcal B$,  each $A\in \GL(n)$
defines a linear map $\phi_A\colon V\to V$.
We obtain an algebraic action 
$$
\GL(n)\times S \to S,\qquad
(A,f)\mapsto A\cdot f\, :=\, f(\phi_A(T_1),\ldots,\phi_A(T_r)).
$$
\end{construction}
 
 For the second step, the idea is to
 determine equations cutting out those matrices in  $\GL(n)$
 that permute the homogeneous components $S_w$ of same dimension
  where~$w\in \Omega_S$.
  As $\Omega_S$ must be fixed by each automorphism,
  it suffices to consider the finite set
  $$
  \aut(\Omega_S)\ :=\ \{\psi\in \aut(K);\ \psi(\Omega_S)=\Omega_S\}\ \subseteq\ \Aut(K).
  $$
  It can be computed by tracking a lattice basis among the set of free parts $q_i^0$ of the~$q_i$,  see~\cite[Remark~3.1]{HaKeWo}.
  
 \begin{algorithm}[Compute $\aut_K(S)$] 
 \label{algo:autKS}
 See~\cite[Algorithm~3.7]{HaKeWo}.
 {\em Input: } the $K$-graded polynomial ring~$S$.
  \begin{itemize}
   \item 
   Determine $\Omega_S = \{w_1,\ldots,w_s\}$.
   Compute a basis $\mathcal B$ as in Construction~\ref{con:action}.

   \item 
   Define the polynomial ring $S' := \KK[Y_{ij};\,1\leq i,j\leq n]$.
   \item 
   Compute an ideal $J\subseteq S'$ 
   whose equations ensure for each $A\in V(J)\subseteq \GL(n)$ the multiplicative condition $A\cdot (f_1f_2) = (A\cdot f_1)(A\cdot f_2)$
   where $f_i\in S$.
   \item 
   Compute $\Aut(\Omega_S)\subseteq \Aut(K)$. 
   Determine the subset $\Gamma_0\subseteq \Aut(\Omega_S)$ of those $B$, that map $\mathcal B_i$
   bijectively to $\mathcal B_j$ where $w_j = B\cdot w_i$.
   \item
   For each $B\in \Gamma_0$, do
   \begin{itemize}
    \item 
   compute an ideal $J_B\subseteq S'$ ensuring that each matrix in
   $V(J_B)\subseteq \GL(n)$ maps the component $S_{w}$ to the component $S_{B\cdot w}$ where $w\in\Omega_S$.  
   \item Redefine $J := J\cdot J_B$.
   \end{itemize}
  \end{itemize}
  {\em Output: } the ideal $J\subseteq S'$. Then $V(J)\subseteq \GL(n)$ is an algebraic subgroup isomorphic to~$\aut_K(S)$.
 \end{algorithm}

 \begin{remark}
 \label{rem:autKS}
\begin{enumerate}
 \item 
 Note that the third step of Algorithm~\ref{algo:autKS}
 is a finite one, see~\cite[Definition~3.4(i)]{HaKeWo} for the precise definition.
 \item 
 The ring $S'$ in Algorithm~\ref{algo:autKS}
 is $K$-graded by defining $\deg(Y_{ij})$ as the degree of the $i$-the element of~$\mathcal B$.
 \item
 The isomorphism $S\to S$ given by $A=(a_{ij})\in V(J)\subseteq \GL(n)$ is as in Construction~\ref{con:action}; explictly, it is given by
 $T_i\mapsto\sum_{j} a_{ij} (\mathcal B_i)_j.$
 \end{enumerate}
 \end{remark}
 
 \begin{example}[\tt{autgradalg.lib} II]
 \label{example:runningII}
  Let us apply Algorithm~\ref{algo:autKS} to Example~\ref{example:running}.
  Here, $\mathcal B= (T_1,\ldots,T_8)$ and all bases $\mathcal B_i = (T_i)$ are one-dimensional.
    Since no weight appears multiple times,
    $\Omega_S = \{q_1,\ldots,q_8\}$.
    Next, the algorithm will compute $\aut(\orig_R)$.
    In our implemenation one can also trigger 
    this step manually if desired:

\begingroup
\footnotesize
\begin{enumerate}
\item[\tt >] \tt{list origs = autGenWeights(Q, TOR);}
\end{enumerate}
\endgroup
\noindent
The result, \texttt{origs}, is a list of four integral matrices (\texttt{intmat}s)
standing for the automorphisms of the generator weights 

\begin{center}
\allowdisplaybreaks
\begin{minipage}{6cm}
\begin{gather}
\aut(\Omega_S)
\ =\ 
\Bigg\{
\id
,
\left[\mbox{\tiny
$\begin{array}{rrrr}
1 & -2 & 0 & 0 \\ 
0 & -1 & 0 & 0 \\ 
0 & 0 & 1 & 0 \\ 
0 & 1 & 0 & 1 \\ 
\end{array}$}
\right]
,\label{autorig}\\
\left[\mbox{\tiny
$\begin{array}{rrrr}
-1 & 2 & 0 & 0 \\ 
0 & 1 & 0 & 0 \\ 
0 & 0 & 1 & 0 \\ 
0 & 1 & 1 & 1 \\ 
\end{array}$}
\right]
,
\left[\mbox{\tiny
$\begin{array}{rrrr}
-1 & 0 & 0 & 0 \\ 
0 & -1 & 0 & 0 \\ 
0 & 0 & 1 & 0 \\ 
0 & 0 & 1 & 1 \\ 
\end{array}$}
\right]
\Bigg\}.
\notag
\end{gather}
\end{minipage}
\qquad 
\begin{minipage}{5cm}
\footnotesize
  \begin{tikzpicture}[scale=2.5]
\coordinate (O) at (0,0,0);
\coordinate (w1) at (0.4,0.48,-.1);
\coordinate (w2) at (0.45,0.54,0.55);
\coordinate (w3) at  (0,0.5,0.5);
\coordinate (w4) at  (0,0.5,-0.5);
\coordinate (w5) at (-0.45,0.46,-0.55);
\coordinate (w6) at (-0.4,0.52,.1);
\coordinate (w7) at (.8,.5,.375);
\coordinate (w8) at (-.8,.5,-.375);
\coordinate (w) at (0,1.1,0);
\coordinate (b1) at (intersection of w8--w2 and w4--w6);
\coordinate (b5) at (intersection of w8--w2 and w6--w7);
\coordinate (b3) at (intersection of w1--w5 and w4--w6);
\coordinate (b4) at (intersection of w4--w2 and w1--w5);
\coordinate (b2) at (intersection of w3--w1 and w2--w6);
\coordinate (b6) at (intersection of w6--w2 and w3--w5);
\coordinate (b7) at (intersection of w3--w5 and w2--w8);
\coordinate (b8) at (intersection of w5--w3 and w6--w4);
\coordinate (a1) at (intersection of w1--w8 and w6--w4);
\coordinate (a2) at (intersection of w1--w8 and w5--w7);
\coordinate (a3) at (intersection of w2--w4 and w5--w7);
\coordinate (a4) at (intersection of w2--w4 and w1--w3);
\coordinate (a5) at (intersection of w1--w3 and w6--w7);
\coordinate (a6) at (intersection of w5--w3 and w6--w7);
\coordinate (a7) at (intersection of w2--w8 and w1--w3);
\coordinate (a8) at (intersection of w2--w4 and w7--w6);
\coordinate (a9) at (intersection of w1--w3 and w5--w7);
\coordinate (d1) at (intersection of w2--w4 and w1--w8);
\coordinate (d2) at (intersection of w5--w7 and w4--w6);
\coordinate (d3) at (intersection of w3--w5 and w1--w8);
\coordinate (d4) at (intersection of w1--w8 and w5--w6);
\coordinate (d5) at (intersection of w2--w8 and w5--w6);
\coordinate (d6) at (intersection of w1--w2 and w6--w7);
\coordinate (d7) at (intersection of w1--w2 and w5--w7);
\fill[color=black!20] (O) -- (w4) -- (w7)--cycle;
\fill[color=black!35] (O) -- (w8) -- (w4)--cycle;
\foreach \w in {1,...,8} {
  \fill (w\w) circle (0.4pt);
}
\fill (0,0,0) circle (0.4pt);
\draw (w1) node[anchor=south]{$q_1$};
\draw (w4) node[anchor=south]{$q_4$};
\draw (w5) node[anchor=south]{$q_5$};
\draw (w6) node[anchor=north east]{$q_6$};
\draw (w7) node[anchor=west]{$q_7$};
\draw (w8) node[anchor=east]{$q_8$};
\draw (O) node[anchor=north]{$(0,0,0)$};
\draw[<->, line width=1.7pt, color=red!65!black, shorten >= .3cm, shorten <= .3cm] (w8) -- (w7);
\draw[<->, line width=1.7pt, color=red!65!black, shorten >= .3cm, shorten <= .3cm] (w4) -- (w3);
\fill[color=black!75] (O) -- (w3) -- (w7)--cycle;
\fill[color=black!65] (O) -- (w3) -- (w8)--cycle;
\draw (w2) node[anchor=south]{$q_2$};
\draw (w3) node[anchor=south]{$q_3$};
\foreach \w in {1,...,8} {
  \fill (w\w) circle (0.4pt);
}
\fill (0,0,0) circle (0.4pt);
\end{tikzpicture}
\end{minipage}
\end{center}
Note that $\aut(\orig_R)$ is isomorphic to the symmetry group 
$\ZZ/2\ZZ\times \ZZ/2\ZZ$ of a $2$-dimensional rhombus.
We now compute $\aut_K(S)$ with the command 
\begingroup
\footnotesize
\begin{enumerate}
\item[\tt >] \tt{def Sprime = autKS(TOR);}
\item[\tt >] \tt{setring Sprime;}
\end{enumerate}
\endgroup
\noindent
A closer inspection shows that \texttt{Sprime} stands for the 
ring $S'=\QQ[Y_1,\ldots,Y_{64},Z]$.
Furthermore, 
a list \texttt{listAutKS} will be exported:
each element is a triple $(A_B, B, J_{B})$
where $B$ runs through the four elements of $\Aut(\Orig_R)$
and $A_B$ is a formal matrix over \texttt{Sprime}
that encodes isomorphisms of~$S$ as in Remark~\ref{rem:autKS}(iii).
For instance, for \texttt{listAutKS[2]},
the sendond entry in the triple $(A_B,B,J_B)$
is the second matrix listed in~\eqref{autorig}
and the matrix $A_B$~is

\begingroup
\footnotesize
\begin{enumerate}
\item[\tt >] \tt{print(listAutKS[2][1]);}
\out{
\mbox{\tiny
$\begin{array}{rrrrrrrr}
Y(1) & 0 & 0 & 0 & 0 & 0 & 0 & 0 \\ 
0 & 0 & 0 & 0 & Y(13) & 0 & 0 & 0 \\ 
0 & 0 & 0 & 0 & 0 & 0 & 0 & Y(24) \\ 
0 & 0 & 0 & 0 & 0 & 0 & Y(31) & 0 \\ 
0 & Y(34) & 0 & 0 & 0 & 0 & 0 & 0 \\ 
0 & 0 & 0 & 0 & 0 & Y(46) & 0 & 0 \\ 
0 & 0 & 0 & Y(52) & 0 & 0 & 0 & 0 \\ 
0 & 0 & Y(59) & 0 & 0 & 0 & 0 & 0
\end{array}$}
}
\end{enumerate}
\endgroup
\noindent
The equations obtained from the zero-entries in $A_B$
and its invertible-condition are stored in the ideal $J_{B}$.
The third entry is:

\begingroup
\footnotesize
\begin{enumerate}
\item[\tt >] \tt{print(listAutKS[2][3]);}
\out{Y(2),\ 
Y(3),\ 
\ldots,\ 
Y(63),\ 
Y(64),\ 
-Y(1)Y(13)Y(24)Y(31)Y(34)Y(46)Y(52)Y(59)Z-1
}
\end{enumerate}
\endgroup
\noindent
Moreover, an ideal \texttt{Iexported}, called $J$ in in Algorithm~\ref{algo:autKS},
is being exported
that is the product over all the ideals~$J_B$ where $B$ runs
through $\aut(\orig_R)$. This means $\Aut_K(S)\cong S'/J$ is isomorphic 
to \texttt{Sprime} modulo \texttt{Iexported}; the degree matrix of \texttt{Sprime}
can be obtained via \texttt{getVariableWeights()}.
 \end{example}

 We come to $\aut_K(R)$.
 Restricting the group action of
 Construction~\ref{con:action}
 to $\aut_K(S)\subseteq \GL(n)$,
 we have an algebraic subgroup 
given as \emph{stabilizer}
$$
\Stab_I(\aut_K(S))
\ := \ 
\{A \in \aut_K(S); \ A \cdot I\,=\,I\}
\ \subseteq \ 
\aut_K(S).
$$
Provided $I_w = \{0\}$ holds for all $w\in \Omega_S$,
in~\cite{HaKeWo} the authors have shown that
we have an isomorphism
\[
\stab_I(\aut_K(S))\ \cong\ 
\aut_K(R).
\]
The final step then is the following.
Define the set $\Omega_I := \{\deg(g_1),\ldots,\deg(g_s)\}$ of ideal generator degrees.
The idea is to compute (linear) equations ensureing that the 
vector spaces $I_u$, where $u\in \Omega_I$, are mapped to one-another.

\begin{algorithm}[Computing $\aut_K(R)$]
\label{algo:autKR}
See~\cite[Algorithm~3.8]{HaKeWo}.
\emph{Input:}
the $K$-graded polynomial ring $S$ and the defining 
ideal $I \subseteq S$ of $R$.
\begin{itemize}
\item
Let $J \subseteq S' := \KK[Y_{ij}; \; 1 \le i,j \le n]$
be the output of Algorithm~\ref{algo:autKS}. 
\item 
Compute $\Omega_I$ and form the $\KK$-vector space $W := \bigoplus_{\Omega_I} S_u$.
\item
For the vector space $I_W = I \cap W \subseteq W$, compute
\begin{itemize}
 \item a $\KK$-basis $(h_1,\ldots,h_l)$ and
 \item a description $I_W = V(\ell_1,\ldots,\ell_m)$
with linear forms $\ell_{i} \in W^*$.
\end{itemize}
\item 
With the $\GL(n)$-action from Construction~\ref{con:action}
and $Y=(Y_{ij})$, we obtain the ideal 
$$
J'
\ := \
\left\< 
\ell_{i}\left(Y\cdot h_j\right); \ 1\leq i\leq m ,\ 1\leq j\leq l
\right\>
\ \subseteq\ S'.
$$
\end{itemize}
\emph{Output:} 
the ideal
$J+J' \subseteq S'$. 
Then $V(J+J')\subseteq \GL(n)$ 
is an algebraic subgroup isomorphic to~$\aut_K(R)$.
\end{algorithm}

\begin{remark}
 \begin{enumerate}
  \item 
  Algorithms~\ref{algo:autKR} and~\ref{algo:autKS} do not make use of Gr\"obner basis computations. However, in \texttt{Singular}, it usually is quicker to compute 
  $J\cap J_B$ instead of $J\cdot J_B$.
  \item 
Computing $G:=\aut_K(R)\subseteq \GL(n)$
with  Algorithm~\ref{algo:autKR}
 enables us to directly compute
 the number of irreducible components 
 $[G:G^0]$ and the dimension of $G$
 by Gr\"obner basis computations.
 \end{enumerate}
\end{remark}

\begin{example}[\tt{autgradalg.lib} III]
\label{example:runningIII}
Continuing Example~\ref{example:runningII},
let us compute $\aut_K(R)$. We first switch back to $S$, 
enter the defining ideal~$I$ for~$R=S/I$
and start the computation of~$\aut_K(R)$:

\begingroup
\footnotesize
\begin{enumerate}
\item[\tt >] \tt{setring S;}
\item[\tt >] \tt{ideal I = T(1)*T(6) + T(2)*T(5) + T(3)*T(4) + T(7)*T(8);}
\item[\tt >] \tt{def Sres = autGradAlg(I, TOR);}
\item[\tt >] \tt{setring Sres;}
\end{enumerate}
\endgroup

The resulting ring \texttt{Sres} is identical to \texttt{Sprime}.
A list \texttt{stabExported} is being exported; 
the interpretation of the entries is identical 
to that of the list \texttt{listAutKS} from Example~\ref{example:runningII}
with the difference, that the ideal part now contains
additional equations describing the stabilizer: for example

\begingroup
\footnotesize
\begin{enumerate}
\item[\tt >] \tt{stabExported[2][3]}
\out{Y(2),\ Y(3),\ 
\ldots
Y(63),\ Y(64),\ -Y(1)Y(13)Y(24)Y(31)Y(34)Y(46)Y(52)Y(59)Z-1,
\\
-Y(24)Y(31)+Y(52)Y(59),\ Y(13)Y(34)-Y(52)Y(59),\ -Y(13)Y(34)+Y(1)Y(46)}
\end{enumerate}
\endgroup
\noindent
Moreover, an ideal \texttt{Jexported} is being exported that is the product over all $J_B$ as before. Then
\texttt{Sres} modulo \texttt{Jexported} is isomorphic to
$\aut_K(R)$.
The grading is obtained as before with
\texttt{getVariableWeights()}.
\end{example}

\section{Application: Mori Dream Spaces}
\label{sec:MDS}

In this section, we shortly recall from~\cite{HaKeWo} 
how the algorithms from the last section
can can be applied to a class of varieties in algebraic geometry.

To a normal algebraic variety $X$ over $\CC$ 
with finitely generated
class group $\Cl(X)$ one can assign 
a $\Cl(X)$-graded $\CC$-algebra, its so-called {\em Cox ring\/}
\[
 \Cox(X)
\ = \
\bigoplus_{[D] \in \Cl(X)} \Gamma(X,\mathcal{O}(D)),
\]
see e.g.~\cite{ArDeHaLa} for details on this theory.
If $X$ is finitely generated, $X$ is called a {\em Mori dream space\/}.
For example, each toric variety or each smooth Fano variety is a 
Mori dream space~\cite{Co, BCHM}.
The Cox ring has strong implications on the underlying Mori dream space.
More precisely, $X$ can be recovered as a good quotient
\begin{gather}
\xymatrix{
\spec(R)
\ar@{}[r]|(.62){=:}
&
{\overline X}
\ar@{}[r]|{\supseteq}
&
{\widehat X}
\ar[r]^{\good H}
&
X
}
\label{eq:MDSseq}
\end{gather}
of an open subset $\widehat X$
by the {\em characteristic quasitorus} $H:=\spec(\KK[K])$.
In fact, $\widehat X$ is determined by 
an ample class $w\in \Cl(X)$.
This opens up a computer algebra based 
approach~\cite{HaKe, Ke:diss} to Mori dream spaces.
In~\cite{ArHaHeLi}, it has been shown that~\eqref{eq:MDSseq} translates to automorphisms of $X$ as follows:
\begin{gather}
 \xymatrix{
    \aut_{\Cl(X)}(\Cox(X))
    \ar@{}[r]|(.61)\cong
 &
   \aut_H(\overline X)
   \ar@{}[r]|{\supseteq}
 &
 \aut_H(\widehat X)
 \ar[r]^{/H}
 &
 \aut(X)
 }
 \label{eq:autmds}
\end{gather}
Here, by $\aut_H(Y)$ we mean the group of 
{\em $H$-equivariant automorphisms} of~$Y$; these 
are pairs $(\phi,\psi)$ with $\phi\colon Y\to Y$ being 
an automorphism of varieties and $\psi\colon H\to H$
an automorphism of affine algebraic groups such that
$\phi(h\cdot y) = \psi(h)\cdot y$ holds for all $h\in H$
and $y\in Y$.
By~\eqref{eq:autmds}, we directly can compute $\aut_H(\overline X)$
with Algorithm~\ref{algo:autKR}. 
In the next proposition, we investigate the symmetries of the list of Fano varieties~\cite{BeHaHuNu}.

\begin{proposition}
\label{prop:fanos}
Let $X_i$ be the non-toric terminal Fano threefold of Picard number one
with an effective two-torus action from the classification~\cite[Theorem~1.1]{BeHaHuNu}.
\begin{enumerate}
 \item 
For all $1\leq i\leq 41$, Algorithm~\ref{algo:autKR} 
is able to compute a presentation of $G_i:=\Aut_H(\b X_i)$
as an affine algebraic subgroup $V(J_i)\subseteq \GL(n_i)$.
\item
Using (i), we list the dimensions $\dim(G_i)$ and the number of components $[G_i:G_i^0]$ of the following~$G_i\subseteq \GL(n_i)$: 

\begingroup
\footnotesize
\begin{longtable}{ccccc}
\myhline
$X_i$ & $\aut(\Omega_S)$ & $\dim(G_i)$ & $[G_i:G_i^0]$ & $\dim(\Aut(X_i))$\\
\myhline 
$X_3$ & $\ZZ/4\ZZ$  & $3$ & $4$ & $2$  \\ \hline 
$X_6$ & $\{1\}$  & $5$ & $ $ & $4$  \\ \hline
$X_7$ & $\{1\}$  & $5$ & $ $ & $4$  \\ \hline
$X_{10}$ & $\{1\}$ & $4$ & $1$ & $3$  \\ \hline
$X_{12}$ & $\{1\}$ & $6$ & $ $ & $5$  \\ \hline
$X_{13}$ & $\{1\}$ & $4$ & $1$ & $3$  \\ \hline
$X_{14}$ & $\{1\}$ & $3$ & $ $ & $2$  \\ \hline
$X_{15}$ & $\{1\}$ & $5$ & $ $ & $4$  \\ \hline
$X_{16}$ & $\{1\}$ & $3$ & $ $ & $2$  \\ \hline
$X_{18}$ & $\{1\}$ & $6$ & $ $ & $5$  \\ \hline
$X_{19}$ & $\{1\}$ & $4$ & $1$ & $3$  \\ \hline
$X_{20}$ & $\{1\}$ & $5$ & $ $ & $4$  \\ \hline
$X_{21}$ & $\{1\}$ & $3$ & $2$ & $2$  \\ \hline
$X_{25}$ & $\{1\}$ & $4$ & $1$ & $3$  \\ \hline
$X_{26}$ & $\ZZ/2\ZZ$ & $3$ & $ $ & $2$  \\ \hline
$X_{28}$ & $\{1\}$ & $4$ & $1$ & $3$  \\ \hline
$X_{33}$ & $\{1\}$ & $6$ & $2$ & $5$  \\ \hline
$X_{34}$ & $\{1\}$ & $6$ & $2$ & $5$  \\ \hline
$X_{36}$ & $\{1\}$ & $5$ & $1$ & $4$  \\ \hline
$X_{37}$ & $\{1\}$ & $4$ & $2$ & $3$  \\ \hline
$X_{38}$ & $\{1\}$ & $4$ & $3$ & $3$  \\ \hline
$X_{39}$ & $\{1\}$ & $3$ & $ $ & $2$  \\ \hline
$X_{40}$ & $\{1\}$ & $3$ & $1$ & $2$  \\ \hline
$X_{42}$ & $\{1\}$ & $3$ & $2$ & $2$  \\ \hline
$X_{45}$ & $\{1\}$ & $4$ & $2$ & $3$  \\ \hline
$X_{46}$ & $\{1\}$ & $4$ & $1$ & $3$  \\ \hline
$X_{47}$ & $\{1\}$ & $3$ & $1$ & $2$  \\
\myhline
\end{longtable}
\endgroup 
\end{enumerate}
\end{proposition}

\begin{proof}
 This is an application of Algorithm~\ref{algo:autKR} and of the \texttt{Singular} commands to compute dimension and absolute components, see for example~\cite{GrePfi}.
 We performed the computations on an older machine (Intel celeron CPU, 4 GB Ram) 
 and cancelled them after several seconds. The files are available at~\cite{autmdslib}.
\end{proof}

In~\cite{HaKeWo}, the authors have also presented algorithms to compute 
$\aut_H(\widehat X)$ and generators for the Hopf algebra $\OO(\aut(X))$.
Both algorithms are also implemented in our library. However, the case  $\OO(\aut(X))$ involves a Hilbert basis computation that usually renders the computation infeasible. We therefore finish this note with an example.

\begin{example}[\tt{autgradalg.lib} IV]
 In Example~\ref{example:runningIII}, the algebra $R$ 
 is the Cox ring of a Mori dream space: fix an ample class, say 
 $w:=(0,0,2)\in K\otimes \QQ$,
 then $R$ and $w$ define a Mori dream space $X=X(R,w)$.
 The characteristic quasitorus is 
 $H = (\KK^*)^3\times \{\pm 1\}$.
 
 In~\ref{example:runningIII}, we   
 have already computed $\aut_H(\overline X)\cong G:=\aut_K(R)$.
 From it, we obtain $\aut_H(\widehat X)$ as follows:
 first, $w$ defines a certain polyhedral cone, the 
 {\em GIT-cone}~$\lambda(w)$.
 Then $\aut_H(\widehat X)$ is obtained from $G$
 by choosing only those
 elements $(A_B,B,J_B)$ of the list \texttt{stabExported}
 where $B\in \aut(\Omega_S)$ fixes $\lambda(w)$.
 In our library, you can compute it with (making use of \texttt{gitfan.lib}~\cite{gitfansymm})
 
\begin{center}
 \begin{minipage}{7cm}
\begingroup
\footnotesize
\begin{enumerate}
\item[\tt >] \tt{intvec w = 1,9,16,0; // drawn in blue }
\item[\tt >] \tt{setring R; // from before}
\item[\tt >] \tt{def RR = autXhat(I, w, TOR);}
\item[\tt >] \tt{setring RR;}
\end{enumerate}
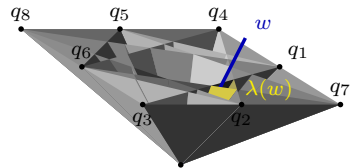
\endgroup
\end{minipage}
 \
  \begin{minipage}{5cm}
  \footnotesize
\begin{tikzpicture}[scale=2.6]
 \coordinate (O) at (0,0,0);
\coordinate (w1) at (0.5,0.5,0);
\coordinate (w2) at (0.5,0.5,0.5);
\coordinate (w3) at  (0,0.5,0.5);
\coordinate (w4) at  (0,0.5,-0.5);
\coordinate (w5) at (-0.5,0.5,-0.5);
\coordinate (w6) at (-0.5,0.5,0);
\coordinate (w7) at (1,.5,.5);
\coordinate (w8) at (-1,.5,-.5);
\coordinate (b1) at (intersection of w8--w2 and w4--w6);
\coordinate (b5) at (intersection of w8--w2 and w6--w7);
\coordinate (b3) at (intersection of w1--w5 and w4--w6);
\coordinate (b4) at (intersection of w4--w2 and w1--w5);
\coordinate (b2) at (intersection of w3--w1 and w2--w6);
\coordinate (b6) at (intersection of w6--w2 and w3--w5);
\coordinate (b7) at (intersection of w3--w5 and w2--w8);
\coordinate (b8) at (intersection of w5--w3 and w6--w4);
\coordinate (a1) at (intersection of w1--w8 and w6--w4);
\coordinate (a2) at (intersection of w1--w8 and w5--w7);
\coordinate (a3) at (intersection of w2--w4 and w5--w7);
\coordinate (a4) at (intersection of w2--w4 and w1--w3);
\coordinate (a5) at (intersection of w1--w3 and w6--w7);
\coordinate (a6) at (intersection of w5--w3 and w6--w7);
\coordinate (a7) at (intersection of w2--w8 and w1--w3);
\coordinate (a8) at (intersection of w2--w4 and w7--w6);
\coordinate (a9) at (intersection of w1--w3 and w5--w7);
\coordinate (d1) at (intersection of w2--w4 and w1--w8);
\coordinate (d2) at (intersection of w5--w7 and w4--w6);
\coordinate (d3) at (intersection of w3--w5 and w1--w8);
\coordinate (d4) at (intersection of w1--w8 and w5--w6);
\coordinate (d5) at (intersection of w2--w8 and w5--w6);
\coordinate (d6) at (intersection of w1--w2 and w6--w7);
\coordinate (d7) at (intersection of w1--w2 and w5--w7);
\fill[color=black!60] (O) -- (w4) -- (w5)--cycle;
\fill[color=black!40] (O) -- (w4) -- (w1)--cycle;
\fill[color=black!70] (O) -- (w4) -- (b3)--cycle;
\fill[color=black!20] (O) -- (w4) -- (b4)--cycle;
\fill[color=black!40] (O) -- (w5) -- (b3)--cycle;
\fill[color=black!30] (O) -- (b3) -- (b4)--cycle;
\fill[color=black!30] (O) -- (b4) -- (w1)--cycle;
\fill[color=black!70] (O) -- (b3) -- (d2)--cycle;
\fill[color=black!20] (O) -- (b4) -- (d1)--cycle;
\fill[color=black!30] (O) -- (w5) -- (d2)--cycle;
\fill[color=black!70] (O) -- (w5) -- (d3)--cycle;
\fill[color=black!60] (O) -- (w5) -- (w8)--cycle;
\fill[color=black!80] (O) -- (w5) -- (d4)--cycle;
\fill[color=black!50] (O) -- (w1) -- (d1)--cycle;
\fill[color=black!80] (O) -- (w1) -- (a9)--cycle;
\fill[color=black!50] (O) -- (d4) -- (d3)--cycle;
\fill[color=black!50] (O) -- (d4) -- (w8)--cycle;
\fill[color=black!80] (O) -- (d4) -- (d5)--cycle;
\fill[color=black!45] (O) -- (w8) -- (d5)--cycle;
\fill[color=black!40] (O) -- (w8) -- (w6)--cycle;
\fill[color=black!50] (O) -- (d5) -- (b1)--cycle;
\fill[color=black!40] (O) -- (w1) -- (w7)--cycle;
\fill[color=black!30] (O) -- (d7) -- (w1)--cycle;
\fill[color=black!50] (O) -- (d7) -- (a9)--cycle;
\fill[color=black!50] (O) -- (d7) -- (w7)--cycle;
\fill[color=black!30] (O) -- (d7) -- (d6)--cycle;
\fill[color=black!60] (O) -- (w7) -- (d6)--cycle;
\fill[color=black!60] (O) -- (d6) -- (a8)--cycle;
\fill[color=black!45] (O) -- (a9) -- (a3)--cycle;
\fill[color=black!80] (O) -- (a9) -- (a4)--cycle;
\fill[color=black!40] (O) -- (d1) -- (a3)--cycle;
\fill[color=black!50] (O) -- (d1) -- (a2)--cycle;
\fill[color=black!35] (O) -- (d2) -- (a2)--cycle;
\fill[color=black!70] (O) -- (d2) -- (a1)--cycle;
\fill[color=black!40] (O) -- (d3) -- (a1)--cycle;
\fill[color=black!60] (O) -- (d3) -- (b8)--cycle;
\fill[color=black!70] (O) -- (b8) -- (a1)--cycle;
\fill[color=black!50] (O) -- (a1) -- (a2)--cycle;
\fill[color=black!20] (O) -- (a2) -- (a3)--cycle;
\fill[color=black!40] (O) -- (a3) -- (a4)--cycle;
\fill[color=black!65] (O) -- (b7) -- (b8)--cycle;
\fill[color=black!90] (O) -- (a4) -- (a5)--cycle;
\fill[color=black!70] (O) -- (a5) -- (b5)--cycle;
\fill[color=black!80] (O) -- (b5) -- (b7)--cycle;
\fill[color=black!50] (O) -- (b1) -- (b8)--cycle;
\fill[color=black!60] (O) -- (b1) -- (b7)--cycle;
\fill[color=black!40] (O) -- (a6) -- (b7)--cycle;
\fill[color=black!60] (O) -- (a6) -- (b5)--cycle;
\fill[color=black!80] (O) -- (a7) -- (a5)--cycle;
\fill[color=black!60] (O) -- (a7) -- (b5)--cycle;
\coordinate (www) at (.5,.82,.45);
\coordinate (www0) at (0,0,0);
\fill[color=yellow!60!black] (O) -- (a8) -- (a4)--cycle;
\draw[color=blue!70!black,line width=1.3pt] (www0) -- (www) node[anchor=south west]{$w$};
\fill[color=yellow!80!black] (O) -- (a8) -- (a5)--cycle;
\fill[color=black!80] (O) -- (a7) -- (b2)--cycle;
\fill[color=black!30] (O) -- (b6) -- (b2)--cycle;
\fill[color=black!40] (O) -- (b6) -- (w3)--cycle;
\fill[color=black!60] (O) -- (b6) -- (a6)--cycle;
\fill[color=black!80] (O) -- (w6) -- (d5)--cycle;
\fill[color=black!70] (O) -- (w6) -- (b1)--cycle;
\fill[color=black!50] (O) -- (w6) -- (a6)--cycle;
\fill[color=black!30] (O) -- (w6) -- (b6)--cycle;
\fill[color=black!40] (O) -- (w6) -- (w3)--cycle;
\fill[color=black!70] (O) -- (b2) -- (w3)--cycle;
\fill[color=black!80] (O) -- (w2) -- (w7)--cycle;
\fill[color=black!50] (O) -- (w2) -- (d6)--cycle;
\fill[color=black!45] (O) -- (w2) -- (a8)--cycle;
\fill[color=black!40] (O) -- (w2) -- (a7)--cycle;
\fill[color=black!30] (O) -- (w2) -- (b2)--cycle;
\fill[color=black!80] (O) -- (w2) -- (w3)--cycle;
\foreach \w in {1,...,8} {
  \fill (w\w) circle (0.4pt);
}
\fill (0,0,0) circle (0.4pt);
\draw (w1) node[anchor=south west]{$q_1$};
\draw (w2) node[anchor=north]{$q_2$};
\draw (w3) node[anchor=north]{$q_3$};
\draw (w4) node[anchor=south]{$q_4$};
\draw (w5) node[anchor=south]{$q_5$};
\draw (w6) node[anchor=south]{$q_6$};
\draw (w7) node[anchor=south]{$q_7$};
\draw (w8) node[anchor=south]{$q_8$};
\draw (O); 
 \draw (a8) node[anchor=west,color=yellow]{$\lambda(w)$};
\end{tikzpicture}
 \end{minipage}
 \end{center} 
Then a list \tt{RES} will be  exported; it is identical to 
the list \tt{stabExported} from Example~\ref{example:runningIII}
with the difference, that it contains only the element \tt{stabExported[1]}
as the other matrices $B$ do not fix~$\lambda(w)$.
The computation of generators for $\OO(\aut(X))$ is not feasible here; in priciple, the command is~\tt{autX(I, w, TOR)}.
\end{example}

\bibliographystyle{abbrv}

\begin{thebibliography}{10}

\bibitem{ArDeHaLa}
I.~Arzhantsev, U.~Derenthal, J.~Hausen, and A.~Laface.
\newblock {\em Cox rings}, volume 144 of {\em Cambridge Studies in Advanced
  Mathematics}.
\newblock Cambridge University Press, Cambridge, 2014.

\bibitem{ArHaHeLi}
I.~Arzhantsev, J.~Hausen, E.~Herppich, and A.~Liendo.
\newblock The automorphism group of a variety with torus action of complexity
  one.
\newblock {\em Mosc. Math. J.}, 14:429--471, 2014.

\bibitem{BeHaHuNu}
B.~Bechtold, J.~Hausen, E.~Huggenberger, and M.~Nicolussi.
\newblock {On terminal Fano 3-folds with 2-torus action}.
\newblock {\em International Mathematics Research Notices}, 5:1563--1602, 2016.

\bibitem{BCHM}
C.~Birkar, P.~Cascini, C.~D. Hacon, and J.~McKernan.
\newblock Existence of minimal models for varieties of log general type.
\newblock {\em Journal of the American Mathematical Society}, 23(2):405--468,
  2010.

\bibitem{gitfansymm}
J.~B\"ohm, S.~Keicher, and Y.~Ren.
\newblock {Computing GIT-fans with symmetry and the Mori chamber decomposition
  of $\overline M_{0,6}$}.
\newblock 2016.
\newblock Preprint. See arXiv:1603.09241.

\bibitem{Co}
D.~A. Cox.
\newblock The homogeneous coordinate ring of a toric variety.
\newblock {\em J. Algebraic Geom.}, 4(1):17--50, 1995.

\bibitem{singular}
W.~Decker, G.-M. Greuel, G.~Pfister, and H.~Sch\"onemann.
\newblock {\sc Singular} {4-1-0} --- {A} computer algebra system for polynomial
  computations.
\newblock \url{http://www.singular.uni-kl.de}, 2016.

\bibitem{GrePfi}
G.-M. Greuel and G.~Pfister.
\newblock {\em A Singular Introduction to Commutative Algebra}.
\newblock Springer Publishing Company, Incorporated, 2nd edition, 2007.

\bibitem{HaKe}
J.~Hausen and S.~Keicher.
\newblock A software package for mori dream spaces.
\newblock {\em LMS Journal of Computation and Mathematics}, 18(1):647--659,
  2015.

\bibitem{HaKeWo}
J.~Hausen, S.~Keicher, and R.~Wolf.
\newblock Computing automorphisms of {M}ori dream spaces.
\newblock {\em Mathematics of Computation}, to appear.

\bibitem{gfan}
A.~N. Jensen.
\newblock {G}fan, a software system for {G}r{\"o}bner fans and tropical
  varieties.
\newblock Available at
  \url{http://home.imf.au.dk/jensen/software/gfan/gfan.html}.

\bibitem{Ke:diss}
S.~Keicher.
\newblock {\em {A}lgorithms for {M}ori {D}ream {S}paces}.
\newblock PhD thesis, Universit\"at T\"ubingen, 2014.

\bibitem{autmdslib}
S.~Keicher.
\newblock \texttt{autmds.lib} -- a library for \texttt{Singular} to compute
  automorphisms of graded algebras, 2017.
\newblock Will be made available at
  \url{http://www.math.uni-tuebingen.de/user/keicher/autgradlib/}.

\bibitem{steidel}
S.~Steidel.
\newblock Gr\"obner bases of symmetric ideals.
\newblock {\em J. Symbolic Comput.}, 54:72--86, 2013.

\end{thebibliography}

\end{document}